\documentclass[11pt]{amsart}

\usepackage[mathscr]{eucal}
\usepackage{amsmath,amssymb,amsfonts,amsthm,enumerate}
\usepackage[colorlinks]{hyperref}

\newtheorem{theorem}{Theorem}[section]

\newtheorem{lemma}[theorem]{Lemma}

\newtheorem{proposition}[theorem]{Proposition}

\theoremstyle{definition}
\newtheorem{claim}{}[theorem]

\newcommand{\N}{\mathbb N}
\newcommand{\Z}{\mathbb Z}

\DeclareMathOperator{\ord}{ord}

\DeclareMathOperator{\supp}{supp}

\newcommand{\la}{\langle}
\newcommand{\ra}{\rangle}
\newcommand{\be}{\begin{equation}}
\newcommand{\ee}{\end{equation}}
\newcommand{\und}{\;\mbox{ and }\;}
\newcommand{\nn}{\nonumber}
\newcommand{\ber}{\begin{eqnarray}}
\newcommand{\eer}{\end{eqnarray}}
\newcommand{\Sum}[2]{\underset{#1}{\overset{#2}{\sum}}}

\newcommand{\vp}{\mathsf v}

\renewcommand{\t}{\, | \,}

\numberwithin{equation}{section}


\begin{document}

\title{On the inverse  problem of the $k$-th Davenport constants for groups of rank  $2$}

\author{Qinghai Zhong}

\address{University of Graz, NAWI Graz,
	Department of Mathematics and Scientific Computing \\
	Heinrichstra{\ss}e 36\\
	8010 Graz, Austria}
\email{qinghai.zhong@uni-graz.at}
\urladdr{https://imsc.uni-graz.at/zhong/}

\thanks{This work was supported by the Austrian Science Fund FWF [grant DOI:10.55776/P36852]}


\begin{abstract}
 For a finite abelian group $G$ and a positive integer $k$, let $\mathsf{D}_k(G)$ denote the smallest integer $\ell$ such that each sequence over $G$ of length at least $\ell$ has $k$ disjoint nontrivial zero-sum subsequences. It is known that $\mathsf D_k(G)=n_1+kn_2-1$ if $G\cong C_{n_1}\oplus C_{n_2}$ is a rank $2$ group, where $1<n_1\t n_2$. We investigate the associated inverse problem for rank $2$ groups, that is,  characterizing the structure of  zero-sum sequences of length $\mathsf D_k(G)$ that can not be partitioned into $k+1$ nontrivial zero-sum subsequences.
\end{abstract}

\maketitle

\section{Introduction}

Let $(G, +, 0)$ be a finite abelian group. By a sequence $S$
over $G$, we mean a finite sequence of terms from $G$ which is unordered, repetition of terms allowed. We say that $S$ is a zero-sum sequence if the sum of its terms equals zero and denote by $|S|$ the length of the sequence.

Let $k$ be a positive integer. 
We denote by $\mathsf{D}_k(G)$ the smallest integer $\ell$ such that every sequence over $G$
of length at least $\ell$ has $k$ disjoint nontrivial zero-sum subsequences.  We call $\mathsf{D}_k(G)$ the $k$-th Davenport constant of $G$, while the Davenport constant $\mathsf D(G)=\mathsf D_1(G)$ is one of the most important zero-sum invariants in Combinatorial Number Theory and,  together with Erd\H os-Ginzburg-Ziv constant,  $\eta$-constant, etc., has been studied a lot (see \cite{Ho-Zh24,Na20,Ad-Gr-Su12,Gi18, Si20,Olson-rk2,Gi-Sc19a,Fa-Ga-Zh11a,Ga-Ha-Pe-Su14,Li20,Bi-Gr-He20,Qu-Li-Te22,Br-Ri18a,Ed-El-Ge-Ku-Ra07,Ca96a,Ga-Lu08a,Ha-Zh19}). 
This variant $\mathsf D_k(G)$ of the Davenport constant was introduced and investigated by F.~Halter-Koch \cite{halterkoch92}, in the context of investigations on the asymptotic behavior of certain counting functions of algebraic integers defined via factorization properties ( see the monograph \cite[Section 6.1]{Ger-book},  and the survey article \cite[Section 5]{gaogeroldingersurvey}).
In 2014, K. Cziszter and M. Domokos (\cite{Cz-Do14,Cz-Do13c}) introduced the generalized Noether Number $\beta_k(G)$ for general groups, which equals $\mathsf D_k(G)$ when $G$ is abelian (see \cite{Cz-Do-Ge16,Cz-Do-Sz17,Cz-Do14a} for more about this direction).
Knowledge of those constants is highly relevant when applying the inductive method to determine or estimate the Davenport constant of certain finite abelian groups (see \cite{delormeetal01,bhowmikschalge07,bhowmikhallschlage,Pl-Sc11}).

In 2010, M. Freeze and W. Schmid (\cite{Fr-Sc10}) showed that for each finite abelian group $G$ we have $\mathsf{D}_k(G)= \mathsf{D}_0(G)+ k\exp(G)$ for some $\mathsf{D}_0(G)\in \mathbb{N}_0$ and all sufficiently large $k$.
In fact, it is known that for groups of rank at most two, and for some other types of groups, an equality of the form $\mathsf{D}_k(G)= \mathsf{D}_0(G)+ k\exp(G)$ for some $\mathsf{D}_0(G)\in \mathbb{N}_0$ holds for all $k$. In particular, for a rank two abelian group $G=C_m\oplus C_n$, where $m\mid n$, we have $\mathsf D_k(G)=m+kn-1$ (\cite[Theorem 6.1.5]{Ger-book}). 
Yet, it fails for elementary $2$ and $3$-groups of rank at least $3$ (see \cite{delormeetal01,bhowmikschalge07}). In general, computing (even bounding) $\mathsf D_k(G)$ is quite more complicated than for $\mathsf D(G)$, in particular for (elementary) $p$-groups, while $\mathsf D(G)$ is know for $p$-groups.

In  zero-sum theory, the associated inverse problems of zero-sum invariants  study the structure of extremal sequences that do not have enough zero-sum subsequences with the prescribed properties. The inverse problems of the Davenport constant, the $\eta$-constant, and the Erd\H os-Ginzburg-Ziv constant are central topics (see \cite{Yu07a,Yu-Ze11b,Sa-Ch07a,Sa-Ch12a,Ga-Pe-Wa11a,GeF21a,Eb-Gr24a,Gr-Li22a, Gr-Li22b,Sc10b,Sc12a,Gi-Sc19b,Gr-uzi}). 
The associated inverse problem of $\mathsf D_k(G)$ is to characterize the maximal length zero-sum sequences that can not be partitioned into $k+1$ nontrivial zero-sum subsequences. In particular, the inverse problem of $\mathsf D(G)$ is to characterize the structure of minimal zero-sum subsequences of length $\mathsf D(G)$, which was accomplished for groups of rank $2$ in a series of papers \cite{Reiher-propB-thesis} \cite{Gao-Ger-propB} \cite{propB-GGG} \cite{Sc10b} \cite{Schlage-case9-propB},
   where a minimal zero-sum sequence is a zero-sum sequence that can not be partitioned into two nontrivial zero-sum subsequences. Those inverse results can be used to construct minimal  generating subsets in Invariant Theory (see \cite[Proposition 4.7]{Cz-Do-Ge16}).

    Let $\mathcal B(G)$ be the set of all zero-sum sequences over $G$. We define
   \[
   \mathcal M_k(G)=\{S\in \mathcal B(G)\colon S \text{ can not be partitioned into $k+1$ nontrivial zero-sum subsequences}\}\,.
   \]
   Then it is easy to see that $\mathsf D_k(G)=\max \{|S|\colon S\in \mathcal M_k(G)\}$.
   In this paper, we investigate the inverse problem of general Davenport constant $\mathsf D_k(G)$ for all rank $2$ groups, that is, to study the structure of sequences of $\mathcal M_k(G)$ of length $\mathsf D_k(G)$. In 2003, Gao and Geroldinger (\cite[Theorem 7.1]{Gao-Ger-propB}) studied the inverse problem of $\mathsf D_k(G)$ for $G=C_n\oplus C_n$ under some assumptions of $G$, which had been confirmed later.  We reformulate this result in the following  and  a proof  will be given in Section 3.
   
\begin{theorem}\label{main1}
	Let  $G = C_{n} \oplus C_{n}$  with $n\ge 2$, let $k\ge 1$, and let $U\in \mathcal B(G)$ with $|U|=\mathsf D_k(G)$. Then  $U \in \mathcal M_k(G)$ 
	if and only if there exists a basis $(e_1,e_2)$ of $G$ such that it has one of the following two forms.
	\begin{itemize}
		\smallskip
		\item[(I)] \[
		U = e_1^{k_1n-1}  \prod_{i\in [1,k_2n]} (x_{i}e_1+e_2), \quad \text{where}
		\]
		\begin{itemize}
			\item[(a)] $k_1,k_2\in \N$ with $k_1+k_2=k+1$,
			\item[(b)] $x_1, \ldots, x_{k_2n}  \in [0, n-1]$ and $x_1 + \ldots + x_{k_2n} \equiv 1 \mod n$.
		\end{itemize}
	\end{itemize}

	\begin{itemize}
			\item[(II)] 
		\[
		U=e_1^{an} e_2^{bn-1} (xe_1+e_2)^{cn-1}  (xe_1+2e_2)\,, \quad \text{where}
		\]
		\begin{itemize}
			\item[(a)] $x\in [2,n-2]$ with $\gcd(x,n)=1$,
			\item[(b)] $a,b,c\geq 1$ with $a+b+c=k+1$.
		\end{itemize}
	
	Note that in this case, we have $k\ge 2$.
	\end{itemize}
\end{theorem}
   
 For general groups, we have the following  main theorem. 
   
\begin{theorem}\label{main}
	Let  $G = C_{n_1} \oplus C_{n_2}$  with $1 < n_1 \mid n_2$ and $n_1<n_2$, let $k\ge 1$, and let $U\in \mathcal B(G)$ with $|U|=\mathsf D_k(G)$. Then  $U \in \mathcal M_k(G)$ 
	if and only if it has one of the following four forms.
	\begin{itemize}
		\smallskip
		\item[(I)] \[
		U = e_1^{\ord (e_1)-1}  \prod_{i\in [1,k\ord (e_2)]} (x_{i}e_1+e_2), \quad \text{where}
		\]
		\begin{itemize}
			\item[(a)] $(e_1, e_2)$ is a basis for $G$ with $\ord(e_1)=n_1$ and $\ord(e_2)=n_2$,
			\item[(b)] $x_1, \ldots, x_{k\ord (e_2)}  \in [0, \ord (e_1)-1]$ and $x_1 + \ldots + x_{k\ord (e_2)} \equiv 1 \mod \ord (e_1)$.
		\end{itemize}
	\end{itemize}

	\begin{itemize}
		\smallskip
		\item[(II)] \[
		U = e_1^{k\ord (e_1)-1}  \prod_{i\in [1,\ord (e_2)]} (x_{i}e_1+e_2), \quad \text{where}
		\]
		\begin{itemize}
			\item[(a)] $(e_1, e_2)$ is a basis for $G$ with $\ord(e_1)=n_2$ and $\ord(e_2)=n_1$,
			\item[(b)] $x_1, \ldots, x_{\ord (e_2)}  \in [0, \ord (e_1)-1]$ and $x_1 + \ldots + x_{\ord (e_2)} \equiv 1 \mod \ord (e_1)$.
		\end{itemize}
	\end{itemize}

	\begin{itemize}
		\item[(III)] \[
		U = g_1^{n_1 - 1}   \prod_{i\in [1,kn_2]} ( -x_{i} g_1 +g_2) \,, \quad \text{where}
		\]
		\begin{itemize}
			\item[(a)] $(g_1, g_2)$ is a generating set of $G$ with $\ord(g_1)>n_1$ and $\ord (g_2) = n_2$,
			\item[(b)] $x_1, \ldots, x_{kn_2} \in [0, n_1-1]$ with $x_1 + \ldots + x_{kn_2} = n_1-1$.
		\end{itemize}
	\end{itemize}

	\begin{itemize}
		\item[(IV)]
		\[
		U = e_1^{s n_1 - 1} \prod_{i\in [1,kn_2 -(s-1)n_1]} ( (1-x_{i}) e_1+e_2) \,, \quad \text{where}
		\]
		\begin{itemize}
			\item[(a)] $(e_1, e_2)$ is a basis of $G$ with $\ord(e_1)=n_2$ and $\ord (e_2) = n_1$,
			\item[(b)] $s \in [2, kn_2/n_1 -1]$,
			\item[(c)]$x_1, \ldots, x_{kn_2 -(s-1)n_1} \in [0, n_1-1]$ with $x_1 + \ldots + x_{kn_2 -(s-1)n_1} = n_1-1$.
		\end{itemize}
	\end{itemize}

\end{theorem}

\section{Notation and Preliminaries}\label{Sec-Prelim}

Our notations and terminology are consistent with \cite{GE} and \cite{Gr-book}. Let $\mathbb{N}$ denote the set of positive integers  and $\mathbb{N}_0=\mathbb{N}\cup\{0\}$. For real numbers $a, b\in \mathbb{R}$, we set the discrete interval $[a, b]=\{x\in \mathbb{Z}\colon a\leq x\leq b\}$.  Throughout this paper, all abelian groups will be written additively, and for $n\in \mathbb{N}$, we denote by $C_n$ a cyclic group with $n$ elements.

Let $G$ be a finite abelian group. It is well-known that $|G|=1$ or $G\cong C_{n_1}\oplus \ldots \oplus C_{n_r}$ with $1<n_1\mid \ldots \mid n_r\in \mathbb{N}$, where $r=\mathsf{r}(G)\in \mathbb{N}$ is the \emph{rank} of $G$, and $n_r={\exp}(G)$ is the \emph{exponent} of $G$. We denote by $|G|$ the \emph{order} of $G$, and $\ord(g)$ the \emph{order} of  $g\in G$. 

Let $\mathcal F(G)$ be the free abelian (multiplicatively written) monoid with basis $G$. Then sequences over $G$ could be viewed as  elements of $\mathcal F(G)$. A  sequence $S\in \mathcal F(G)$ could be written as 
$$S=g_1\cdot \ldots \cdot g_l=\prod_{g\in G}g^{\mathsf v_{g}(S)}\,,$$
where  $\mathsf v_{g}(S)\in \mathbb{N}_0$ is the multiplicity of $g$ in $S$. 
We call
\begin{itemize}
	\item $\supp(S)=\{g\in G\colon \mathsf v_g(S)>0\}\subset G$ the \emph{support} of $S$, and
	\item $\sigma(S)=\sum_{i=1}^{l}g_i=\sum_{g\in G}\mathsf v_g(S)g\in G$ the \emph{sum} of $S$. 	
\end{itemize}
Let $t\in \N$. We denote  $$\Sigma_{\le t}(S)=\left\{\sum_{i\in I}g_i\colon I\subseteq [1, l] \mbox{ with } 1\leq |I|\leq t\right\}\,.$$

A sequence $T\in \mathcal F(G)$ is called a subsequence of $S$ if $\mathsf v_{g}(T)\leq \mathsf v_{g}(S)$ for all $g\in G$, and denoted by $T\mid S$. If $T\mid S$, then we denote
$$T^{-1} S=\prod_{g\in G}g^{\mathsf v_g(S)-\mathsf v_g(T)}\in \mathcal F(G)\,.$$
Let $T_1, T_2\in \mathcal F(G)$.  We set
$$T_1 T_2=\prod_{g\in G}g^{\mathsf v_{g}(T_1)+\mathsf v_{g}(T_2)}\in \mathcal{F}(G)\,.$$
If $T_1,\ldots, T_t\in \mathcal F(G)$ such that $T_1\cdot\ldots\cdot T_t\t S$, where $t\ge 2$, then 
 we say $T_1,\ldots,T_t$ are disjoint subsequences of $S$.

Every map of abelian groups $\phi: G\rightarrow H$ extends to a map from the sequences over $G$ to the sequences over $H$ by setting $\phi(S)=\phi(g_1)\cdot \ldots \cdot \phi(g_l)$. If $\phi$ is a homomorphism, then $\phi(S)$ is a zero-sum sequence if and only if $\sigma(S)\in \mathsf{ker}(\phi)$.

We denote by $\mathsf E(G)$ the Gao's constant which is the smallest integer $\ell$ such that every sequence over $G$ of length $\ell$ has a zero-sum subsequence of length $|G|$ and by $\eta(G)$ the smallest integer $\ell$ such that every sequence over $G$ of length $\ell$ has a zero-sum subsequence $T$ of length $1\le |T|\le \exp(G)$. 
Let $\mathsf d(G)$ be the maximal length of a sequence over $G$ that has no zero-sum subsequence. Then it is easy to see that $\mathsf d(G)=\mathsf D(G)-1$.
The following result is well-known and we may use it without further mention.

\begin{lemma}\label{le-E}
	Let $G$ be a finite abelian group. Then $\mathsf E(G)=|G|+\mathsf d(G)\le 2|G|-1$.
 	\end{lemma}
\begin{proof}
	See \cite[Propositions 5.7.9.2 and 5.1.4.4]{Ger-book}.
\end{proof}

We also need the following lemmas.
\begin{lemma}\label{le-cyc}
	Let $G$ be a finite abelian group. If $\mathsf D(G)=|G|$, then $G$ is cyclic and for every minimal zero-sum sequence $S$ over $G$ of length $|G|$,  there exists $g\in G$ with $\ord(g)=|G|$ such that $S=g^{|G|}$.
\end{lemma}
\begin{proof}
	Let $n=\exp(G)$.
	By \cite[Theorem 5.5.5]{Gr-book}, we have $\mathsf D(G)\le n+n\log\frac{|G|}{n}$, whence $\mathsf D(G)=|G|$ implies that $G$ is cyclic. The remaining assertion follows from \cite[Theorem 5.1.10.1]{Ger-book}.
\end{proof}

\begin{lemma}\label{le-subgroup}
	Let $G$ be a finite abelian group and let $H\subset G$ be a proper subgroup. Then $\mathsf D_k(H)<\mathsf D_k(G)$ for all $k\in \N$.
\end{lemma}
\begin{proof}
	The assertion follows from \cite[Lemma 6.1.3]{Ger-book}.
\end{proof}

\begin{theorem}\label{th-gen}
Let  $G = C_{n_1} \oplus C_{n_2}$  with $n_1 \mid n_2$, where $n_1,n_2\in \N$, and let $k\in \N$. Then $\eta(G)=2n_1+n_2-2$ and $\mathsf D_k(G)=n_1+kn_2-1$. In particular, $\mathsf D(G)=n_1+n_2-1$.
\end{theorem}
\begin{proof}
	The assertion follows from \cite[Theorems 5.8.3 and 6.1.5]{Ger-book}.
\end{proof}

\begin{theorem}  \label{inverse}
	Let  $G = C_{n} \oplus C_{mn}$  with $n\geq 2$  and $m \ge 1$.  A sequence $S$
	over $G$ of length $\mathsf D (G) = n+mn-1$ is a minimal zero-sum
	sequence if and only if it has one of the following two forms{\rm
		\,:}
	\begin{itemize}
		\medskip
		\item[(I)] \[
		S = e_1^{\ord (e_1)-1} \prod_{i=1}^{\ord (e_2)}
		(x_{i}e_1+e_2),
		\]where
		\begin{itemize}\item[(a)] $\{e_1, e_2\}$ is a basis of $G$,
			\item[(b)] $x_1, \ldots, x_{\ord (e_2)}  \in
			[0, \ord (e_1)-1]$ and $x_1 + \ldots + x_{\ord (e_2)} \equiv 1
			\mod \ord (e_1)$. \end{itemize} In this case, we say that $S$ is of type I(a) or I(b) according to whether $\ord(e_2)=n$ or $\ord(e_2)=mn>n$.
		
		\medskip
		\item[(II)] \[
		S = f_1^{sn - 1} f_2^{(m-s)n+\epsilon}\prod_{i=1}^{n-\epsilon} ( -x_{i} f_1 +
		f_2),
		\] where
		\begin{itemize}
			\item[(a)] $\{f_1, f_2\}$ is a generating set for  $G$ with $\ord (f_2) =
			mn$ and $\ord(f_1)>n$,
			\item[(b)] $\epsilon\in [1,n-1]$  and
			$s \in [1, m-1]$,
			\item[(c)] $x_1, \ldots, x_{n-\epsilon} \in [1, n-1]$ with $x_1 + \ldots + x_{n-\epsilon} = n-1$,  \item[(d)] either  $s=1$ or
			$nf_1 = nf_2$, with both holding when $m=2$, and
			\item[(e)] either $\epsilon\geq 2$  or $nf_1\neq nf_2$.\end{itemize} In this case, we say that $S$ is of type II.
	\end{itemize}
\end{theorem}

\begin{proof}
	The characterization of minimal zero-sum sequences of maximal length over groups of rank two was done in a series of papers by   Gao, Geroldinger, Grynkiewicz, Reiher, and Schmid. For the formulation used above we refer to \cite[Main Proposition 5.4]{Ge-Gr-Yu15}.
\end{proof}

\begin{lemma}\label{le-epi}
	Let $G$ be a finite abelian group, let $H$ be a cyclic subgroup of $G$, and let $\varphi\colon G\rightarrow G/H$ be the canonical epimorphism.
	If $S\in \mathcal M_k(G)$, then $\varphi(S)\in \mathcal M_{k|H|}(G/H)$.
\end{lemma}

 \begin{proof}
 	Suppose $S\in \mathcal M_k(G)$. Assume to the contrary that $\varphi(S)\not\in \mathcal M_{k|H|}(G/H)$. Then we can decompose $S=T_1\cdot\ldots\cdot T_{k|H|+1}$ such that $\varphi(T_i)$, $i\in [1, k|H|+1]$, are nontrivial zero-sum sequences. Therefore $S_{\sigma}:=\sigma(T_1)\cdot\ldots\cdot \sigma(T_{k|H|+1})$ is a zero-sum sequence over $H$ with length $k|H|+1>\mathsf D_k(H)$. It follows by the definition of $\mathsf D_k(H)$ that $S_{\sigma}$ and hence $S$ are both a product of $k+1$ nontrivial zero-sum subsequences, a contradiction to $S\in \mathcal M_k(G)$.
 \end{proof}

\section{Proof of main theorems}\label{sec-k=n-1}

\begin{proposition}\label{pr-key}
	Let $G$ be a finite abelian group of rank at most $2$, let $k\in \N$, and let $S$ be a zero-sum sequence over $G$ of length $\mathsf D_k(G)$. Then $S\in \mathcal M_k(G)$ if and only if $0\not\in \Sigma_{\le \exp(G)-1}(S)$.
\end{proposition}
\begin{proof}
	Suppose $0\not\in \Sigma_{\le \exp(G)-1}(S)$. Assume to the contrary that $S=T_1\cdot\ldots\cdot T_{k+1}$, where $T_i$ is a nontrivial  zero-sum subsequence for each $i\in [1,k+1]$. Then $|T_i|\ge \exp(G)$ for each $i\in [1,k+1]$, ensuring $\mathsf D_k(G)=|S|\ge (k+1)\exp(G)$. If $G$ is cyclic, then $\mathsf D_k(G)=k\exp(G)<(k+1)\exp(G)$ (by Theorem \ref{th-gen}), a contradiction. If $\mathsf r(G)=2$, then $\mathsf D_k(G)=k\exp(G)+|G|/\exp(G)-1<(k+1)\exp(G)$ (by Theorem \ref{th-gen}), a contradiction.
	
	Suppose $S\in \mathcal M_k(G)$. Assume to the contrary that $0\in \Sigma_{\le \exp(G)-1}(S)$. 
	Then $S$ has a zero-sum subsequence $T$ with $1\le |T|\le \exp(G)-1$. If $k=1$, then it follows from $|S|=\mathsf D(G)>\exp(G)-1$ that $S\not\in \mathcal A(G)$, a contradiction. Thus we may assume that $k\ge 2$ and hence  $T^{-1}S\in \mathcal M_{k-1}(G)$. If $G$ is cyclic, then Theorem \ref{th-gen} implies  $$(k-1)\exp(G)+1= \mathsf D_k(G)-(\exp(G)-1)\le |T^{-1}S|\le \mathsf D_{k-1}(G)=(k-1)\exp(G)\,,$$ a contradiction.
	If $\mathsf r(G)=2$, then Theorem \ref{th-gen} implies
	 \begin{align*} &(k-1)\exp(G)+|G|/\exp(G)= \mathsf D_k(G)-(\exp(G)-1)\\
	 	\le &|T^{-1}S|\le \mathsf D_{k-1}(G)=(k-1)\exp(G)+|G|/\exp(G)-1\,,
	 \end{align*}
	a contradiction.
\end{proof}

We first investigate the associated inverse problem  for cyclic groups.

\begin{theorem}\label{le-cyclic}
	Let $G$ be cyclic, let $k\in \N$, and let $S$ be a zero-sum sequence over $G$ of length $\mathsf D_k(G)$. Then $S\in \mathcal M_k(G)$ if and only if there exists $g\in G$ with $\ord(g)=|G|$ such that $S=g^{k|G|}$.
\end{theorem}
\begin{proof}
	Note that $\mathsf D_k(G)=k|G|$ (by Theorem \ref{th-gen}). If $S=g^{k|G|}$ for some $g\in G$ with $\ord(g)=|G|$, then the minimal zero-sum subsequence of $S$ can only be of the form $g^{|G|}$, whence $S$ is a product of at most $k$ zero-sum subsequnces. It follows that $S\in \mathcal M_k(G)$.
	
	Suppose $S\in \mathcal M_k(G)$. Let $T$ be a minimal zero-sum subsequence of $S$. By Proposition \ref{pr-key}, we have  $\exp(G)\le |T|\le \mathsf D(G)$, whence $|T|=|G|$ (since $G$ is cyclic). It follows from Lemma \ref{le-cyc} that there exists $g\in G$ with $\ord(g)=|G|$ such that $T=g^{|G|}$. Assume to the contrary that there exists $h\t T^{-1}S$ such that $h=sg$ with $s\in [2,n]$, ensuring that $g^{|G|-s}h$ is a zero-sum subsequence of $S$ with length $|G|-s+1\le |G|-1$, a contradiction to Proposition \ref{pr-key}. Therefore $\supp(T^{-1}S)\subset \{g\}$ and hence $S=g^{k|G|}$.
\end{proof}

Next, we prove Theorem \ref{main1}
 which could be handled easily by  Proposition \ref{pr-key} and \cite[Theorem 7.1]{Gao-Ger-propB}.
\begin{proof}[Proof of Theorem \ref{main1}]
	If $U$ is of type I, then since $\supp(U)\subset \{e_1\}\cup \langle e_1\rangle+e_2$ and $\ord(e_1)=\ord(e_2)=n$, we obtain that $0\not\in \Sigma_{\le n-1}(U)$. It follows from Proposition \ref{pr-key} that $U\in \mathcal M_k(G)$. Suppose $U$ is of type II. Assume to the contrary that there exists a nontrivial zero-sum subsequence $T$ of $U$ with $|T|\le n-1$. If $xe_1+2e_2\nmid T$, then $\supp(T)\subset \{e_1\}\cup \langle e_1\rangle+e_2$ and hence  $|T|\ge n$, a contradiction. Thus $T=(xe_1+2e_2)e_1^{\alpha}e_2^{\beta}(xe_1+e_2)^{\gamma}$ for some $\alpha,\beta,\gamma\in \N_0$, whence $2+\beta+\gamma\equiv 0\mod n$ and $x(1+\gamma)+\alpha\equiv 0\mod n$. Since $|T|=1+\alpha+\beta+\gamma\le n-1$, we obtain that $\alpha=0$, $\beta+\gamma=n-2$, and $n\mid x(1+\gamma)$. It follows from $\gcd(x,n)=1$ that $n\t 1+\gamma$, a contradiction to $\gamma+\beta=n-2$. Thus $0\not\in \Sigma_{\le n-1}(U)$. It follows from Proposition \ref{pr-key} that $U\in \mathcal M_k(G)$.

	Suppose $U\in \mathcal M_k(G)$. By 	\cite{Gi18}, the group $G=C_n\oplus C_n$ has Property B (see \cite[Chapter 5]{GE} for the definition of Property B) and by \cite[Theorem 6.7.2]{gaogeroldingersurvey}, every sequence over $G$ of length $3n-2$ has a zero-sum subsequence of length $n$ or $2n$. Thus all the assumptions of \cite[Theorem 7.1]{Gao-Ger-propB} are fulfilled and hence the assertions follows from \cite[Theorem 7.1]{Gao-Ger-propB}.
\end{proof}

\begin{lemma}\label{le-key}  Let $G=C_n\oplus C_n$ with $n\ge 2$ and let $k\ge 2$. If  $S\in\mathcal F(G)$  is a zero-sum sequence with $|S|=(k+1)n-1$ and $0\notin \Sigma_{\leq n-1}(S)$, then  there is a basis $(e_1,e_2)$ for $G$ such that either
	\begin{itemize}
		\item[1.] $\supp(S)\subseteq \{e_1\}\cup \big(\la e_1\ra+e_2\big)$ and $\vp_{e_1}(S)\equiv -1\mod n$, or
		\item[2.] $S=e_1^{an} e_2^{bn-1} (xe_1+e_2)^{cn-1}  (xe_1+2e_2)$ for some $x\in [2,n-2]$ with $\gcd(x,n)=1$, and some  $a,b,c\geq 1$ with $k+1=a+b+c$.
	\end{itemize}
\end{lemma}

\begin{proof}
	By Theorem \ref{th-gen}, we have $\mathsf D_k(G)=(k+1)n-1$.
	Since $0\notin \Sigma_{\leq n-1}(S)$, it follows from 
	 Proposition \ref{pr-key} that $S\in \mathcal M_k(G)$. Now the assertion follows from Theorem \ref{main1}. Moreover, there is a direct proof of this lemma under the assumption of $G=C_n\oplus C_n$ having Property B (see \cite[Lemma 3.2]{Gr21}).
\end{proof}

The following lemma shows two special cases of Theorem \ref{main}.

\begin{lemma}\label{sch-I}
	Let  $G = C_{n_1} \oplus C_{n_2}$  with $1 < n_1 \mid n_2$ and $n_1<n_2$, let $k\ge 2$, and let  $U \in \mathcal M_k(G)$ with  $|U|=\mathsf D_k (G)$. 
	\begin{enumerate}
		\item[1.] If there is some $e_1\in \supp(U)$ such that $\ord(e_1)=n_1$ and $\vp_{e_1}(U)\geq n_1-1$, then 
		there exists $e_2\in G$ with  $\ord(e_2)=n_2$ such that  $(e_1,e_2)$ is a basis of $G$ and
		\[U = e_1^{n_1-1}  \prod_{i\in [1,kn_2]} (x_{i}e_1+e_2)\,,
		\]
		where $x_1, \ldots, x_{kn_2}  \in [0, n_1-1]$ and $x_1 + \ldots + x_{kn_2} \equiv 1 \mod n_1$.
		
		\item[2.] If there is some $e_2\in \supp(U)$ such that  $\ord(e_2)=n_2$ and $\vp_{e_2}(U)\geq kn_2-1$, then 
		there exists $e_1\in G$ with $\ord(e_1)=n_1$  such that $(e_1,e_2)$ is a basis of $G$ and
		\[U = e_2^{kn_2-1}  \prod_{i\in [1,n_1]} (e_1+x_{i}e_2)\,,
		\]
		where $x_1, \ldots, x_{n_1}  \in [0, n_2-1]$ and $x_1 + \ldots + x_{n_1} \equiv 1 \mod n_2$.
	\end{enumerate}
\end{lemma}

\begin{proof}			
	1.	Suppose there exists $e_1\in \supp(U)$ with $\ord(e_1)=n_1$ and $\mathsf v_{e_1}(U)\ge n_1-1$. 
		Let $$H=\la e_1\ra$$ and let $\phi_H:G\rightarrow G/H$ be the canonical epimorphism. 
		Define $T$ by \be\label{U-hug}U=e_1^{n_1-1} T, \text{ where }T\in \mathcal F(G).\ee
		Then $\phi_H(T)$ is zero-sum over $G/H$ of length $\mathsf D_k(G)-(n_1-1)=kn_2$ (by Theorem \ref{th-gen}).
		Assume to the contrary that $0\in \Sigma_{\le n_2-n_1}(\phi_H(T))$. Then there exists a nontrivial subsequence $T'$ of $T$ with $|T'|\le n_2-n_1$  such that $\sigma(T')=se_1$ for some $s\in [1,n_1]$.
		It follows that $e_1^{n_1-s}T'$ is zero-sum of length $n_1-s+|T'|\le n_1-1+n_2-n_1\le n_2-1$, a contradiction to Proposition \ref{pr-key}. Thus $0\not\in \Sigma_{\le n_2-n_1}(\phi_H(T))$.

		By Lemma \ref{le-E}, we have $\mathsf E(G/H)\le 2|G/H|-1=2n_2-1$ and by repeatedly using this result, we can factorize  $T=T_1\cdot\ldots\cdot T_k$ such that $|T_i|=n_2$ and $\phi_H(T_i)$ is zero-sum for every $i\in [1,k]$.
		If there exists $i\in [1,k]$ such that $\phi_H(T_i)$ is not minimal, then $T_i=T_i^{(1)}T_i^{(2)}$ with $|T_i^{(1)}|\ge |T_i^{(2)}|\ge 1$ such that both $\phi_H(T_i^{(1)})$ and $\phi_H(T_i^{(2)})$ are zero-sum, whence $|T_i^{(2)}|\le \frac{n_2}{2}\le n_2-n_1$, a contradiction to $0\not\in \Sigma_{\le n_2-n_1}(\phi_H(T))$.
		Thus for each $i\in [1,k]$, the sequence $\phi_H(T_i)$ is a minimal zero-sum subsequence of length $n_2=|G/H|$, ensuring by Lemma \ref{le-cyc} that $G/H$ must be cyclic.
		It follows from Lemma \ref{le-cyc} that there exists $e_2\in G$  such that  $\phi_H(e_2)$ is a generator of $G/H$ and $\phi_H(T_1)=\phi_H(e_2)^{n_2}$.
		Assume that there exists $j\in [2,k]$ such that $\phi_H(T_j)\neq \phi_H(e_2)^{n_2}$, then there exists $s\in [2, n_2-1]$ with $\gcd(s,n_2)=1$ such that $\phi_H(T_j)= (s\phi_H(e_2))^{n_2}$. Note that $s\ge 2$. By letting $t\in \N$ be minimal such that $t(s-1)\ge n_1$, we have $\phi_H(e_2)^{n_2-ts}(s\phi_H(e_2))^t\t \phi_H(T_1T_j)$ is zero-sum of length $n_2-ts+t\le n_2-n_1$, a contradiction to $0\not\in \Sigma_{\le n_2-n_1}(\phi_H(T))$. Therefore $\phi_H(T)=\phi_H(e_2)^{kn_2}$.
		Moreover, $\ord(\phi_H(e_2))=n_2$ ensures that $\ord(e_2)$ is a multiple of $n_2=\exp(G)$, which is the maximal order of an element from $G$. This forces $\ord(e_2)=\ord(\phi_H(e_2))=n_2$, which combined with $G=\la e_1,e_2\ra$ and $\ord(e_1)=n_1$ ensures that $G=\la e_1\ra\oplus \la e_2\ra$ with $\ord(e_2)=n_2$.

		Let $\pi_2:G\rightarrow \la e_2\ra$ be the projection homomorphism (with kernel $H=\la e_1\ra$) given by $\pi_2(xe_1+ye_2)=ye_2$. Since we now know $H=\la e_1\ra$ is a direct summand in $G$, we can identify $\pi_2$ with $\phi_H$, whence $\pi_2(T)=e_2^{kn_2}$, ensuring
		$\supp(T)\subset \la e_1\ra+e_2$. Combined with \eqref{U-hug}, the assertion now readily follows from  $U$ being zero-sum.

	2.	Suppose there exists $e_2\in \supp(U)$ with $\ord(e_2)=n_2$ and $\mathsf v_{e_2}(U)\ge kn_2-1$. Then 
		\[
		U=e_2^{kn_2-1}T, \text{ where } T\in \mathcal F(G) \text{ with }|T|=n_1\,.
		\]
		Since $e_2^{(k-1)n_2}$ is a product of $k-1$ zero-sum subsequences of length $n_2$, 
		it follows from $U\in \mathcal M_k(G)$ that $e_2^{n_2-1}T$ must be a minimal zero-sum sequence. Let $$H=\la e_2\ra$$ and since $\exp(G)=n_2$, we have $H$ is a direct summand in $G$ and hence there exists $e_1\in G$ with $\ord(e_1)=n_1$ such that $G=H\oplus \langle e_1\rangle$.
		Let $\pi_1:G\rightarrow \la e_1\ra$ be the projection homomorphism (with kernel $H=\la e_2\ra$) given by $\pi_2(xe_1+ye_2)=xe_1$.
		Then $\pi_1(T)$ is zero-sum over $G/H$ of length $n_1$.
		Assume that $\pi_1(T)$ is not minimal. Then $T=T^{(1)}T^{(2)}$ with both  $\pi_1(T^{(1)})$ and $\pi_1(T^{(2)})$  nontrivial zero-sum. Say $\sigma(T^{(1)})=se_2$ for some $s\in [1, n_2]$. Then $e_2^{n_2-s}T^{(1)}$ is a proper nontrivial zero-sum subsequence of $e_2^{n_2-1}T$, a contradiction. Thus $\pi_1(T)$ is a minimal zero-sum sequence over $\langle e_1\rangle$ of length $n_1$, and hence there exists $s\in [1, n_1-1]$ with $\gcd(s, n_1)=1$ such that $\pi_1(T)=(se_1)^{n-1}$. By replacing the basis $(e_1,e_2)$ with $(se_1,e_2)$, the assertion now readily follows from  $U$ being zero-sum.	
\end{proof}

\bigskip
Now we are ready to prove our main Theorem \ref{main}.

\begin{proof}[Proof of Theorem \ref{main}]
	Let $$n=n_1\;\und\; n_2=mn,\quad\mbox{ with $m\geq 2$}.$$
	Then $$G=C_n\oplus C_{mn}.$$ 
	
	Suppose $k=1$. By Theorem \ref{inverse}, it suffices to show that 
	type II sequences in Theorem \ref{inverse} is equivalent to type III and type IV sequences in Theorem \ref{main}.
	
	Let $S=f_1^{sn - 1} f_2^{(m-s)n+\epsilon}\prod_{i=1}^{n-\epsilon} ( -x_{i} f_1 +
	f_2)$ be a type II sequence in Theorem \ref{inverse}. 
	If $s=1$, then it is easy to see that $S$ is of type III in Theorem \ref{main}. If $s\ge 2$, then II.(d) in Theorem \ref{inverse} implies that $nf_1=nf_2$. Since $(f_1,f_2)$ is a generating set with $\ord(f_2)=mn$, we obtain  $(f_2-f_1,f_1)$ is basis of $G$ with $\ord(f_2-f_1)=n$ and $\ord(f_1)=mn$. Set $g_1=f_1$ and $g_2=f_2-f_1$. Then 
	\[
	S=g_1^{sn - 1} (g_1+g_2)^{(m-s)n+\epsilon}\prod_{i=1}^{n-\epsilon} ( -x_{i} g_1 +
	g_1+g_2)=g_1^{sn - 1} \prod_{i=1}^{(m-s+1)n} ( (1-x_{i}) g_1 + g_2)\,,
	\]
	where $x_i=0$ for all $i\in [n-\epsilon+1, (m-s+1)n]$,
	whence $S$ is of type IV in Theorem \ref{main}. 
	
	For the inverse, let $S$ be a type III or type IV sequence in Theorem \ref{main}. By letting $n-\epsilon$ be the number of $x_i$'s that is not zero, it is to easy to see that $S$ is a type II sequence in Theorem \ref{inverse}.

	 Now we assume that $k\ge 2$.
	Since $\mathsf D_k(G)=kn_2+n_1-1$ by Theorem \ref{th-gen}, we see that all sequences given in (I), (II), (III), or (IV) have length $\mathsf D_k(G)$.
	It is straightforward to check that 
	any sequence $U$ satisfying the conditions given in (I) or (II) 
	has $0\not\in \Sigma_{\le mn-1}(U)$, whence $U\in \mathcal M_k(G)$ follows from Proposition \ref{pr-key}.
	 Let us next verify that type III and type IV sequences $U$ have $0\not\in \Sigma_{\le mn-1}(U)$, and then $U\in \mathcal M_k(G)$ follows from Proposition \ref{pr-key}.
	
	Let $U$ be a type III sequence. Consider a nontrivial minimal zero-sum subsequence $T\mid U$. It is sufficient to show $|T|\ge n_2$.
	 After renumbering if necessary, we may assume that $T=g_1^u\prod_{i=1}^{v}(-x_ig_1+g_2)$, where $u\in [0, n_1-1]$ and $v\in [0, kn_2]$. Thus $0=\sigma(T)=(u-\sum_{i=1}^vx_i)g_1+vg_2$.
	Since $(g_1,g_2)$ is a generating set with $\ord(g_2)=n_2$, we obtain $u-\sum_{i=1}^vx_i$ is a multiple of $n_1$. 
	 It follows from $|u-\sum_{i=1}^vx_i|\in [0, n-1]$ that $u-\sum_{i=1}^vx_i=0$ and hence $v$ is a multiple of $\ord(g_2)$. Since $v=0$ implies $u=0$ and hence $|T|=u+v=0$, it follows from $T$ is nontrivial that $v\ge \ord(g_2)=n_2$ and hence $|T|\ge v\ge n_2$.
	
	Let $U$ be a type IV sequence. Consider a nontrivial minimal zero-sum subsequence $T\mid U$. It is sufficient to show $|T|\ge n_2$.
	 Suppose $$T=e_1^a \prod_{i\in I}((1-x_i)e_1+e_2)$$ for some $a\in [0,sn-1]$ and nonempty $I\subset [1,(km-s+1)n]$  with $n\mid |I|$. 
	By considering the sum of $e_1$-coordinates, it follows that $a+|I|-\sum_{i\in I}x_i\equiv 0\mod mn$, and hence $a\equiv \sum_{i\in I}x_i \mod n$. Set $|I|=s_1n$ and $a=s_2n+\sum_{i\in I}x_i$, where $s_1\in [1, km-s+1]$ and $s_2\in [0, s-1]$.
	Then $(s_1+s_2)n=a+|I|-\sum_{i\in I}x_i\equiv 0\mod mn$, whence $s_1+s_2\equiv 0 \mod m$. 
	It follows from $$ s_1+s_2\ge 1\ \text{ and }\ |T|=a+|I|=(s_1+s_2)n+\sum_{i\in I}x_i\le \mathsf D(G)=mn+n-1 \text{(by Theorem \ref{th-gen})}$$ that $s_1+s_2=m$ and $|T|\ge mn=n_2$.

		\bigskip
	It remains to show that every sequence in $\mathcal M_k(G)$ must have the form either given by (I), (II), (III), or (IV).

	\bigskip
	
	Let $U\in \mathcal M_k(G)$  of length $|U|=kmn+n-1$
	and suppose 
	\begin{equation}\label{assume}
		U\ \  \text{does not have the form of type I  or II.}
		\end{equation}
	We need to show that $U$ has the form of type III or IV.
	
	Let
	$\varphi:G\rightarrow G$ be a homomorphism with
	$$\varphi(G)=\mathsf {im} \varphi \cong C_n\oplus C_n\;\und\; \ker \varphi\cong C_m.$$ For instance, if $(e_1,e_2)$ were a basis for $G$ with $\ord(e_1)=n$ and $\ord(e_2)=mn$, then the map $xe_1+ye_2\mapsto xe_1+yme_2$ is one such a map.
	
	We define a \textbf{block decomposition}
	\index{block decomposition} of $U$ to be
	a tuple $W=(W_0,W_1,\ldots,W_{km-1})$, where
	$$U=W_0
	W_1 \cdot \ldots\cdot W_{km-1}$$ with each $\varphi(W_i)$ a nontrivial zero-sum for $i\in [0,km-1]$.
	
		\begin{claim}\label{sch-A1}
		Let $W=(W_0,\ldots, W_{km-1})$ be a block decomposition of $U$. 
		\begin{enumerate}
			\item[1.] For all $i\in [0,  km-1]$, we have $\varphi(W_i)$ is minimal, $\sigma(W_i)$ is a generator of $\ker(\varphi)$, and 
			$$\sigma(W_0)\cdot\ldots\cdot \sigma(W_{km-1})=\sigma(W_0)^{km}\,.$$
			
			\item[2.] If there exist subsequences $S\t W_i$ and $T\t W_j$ with $i\neq j$ such that $\sigma(\varphi(S))=\sigma(\varphi(T))$, then $\sigma(S)=\sigma(T)$.
			
			\item[3.] If there are distinct blocks $W_i$ and $W_j$ having terms $g\in \supp(W_i)$ and $h\in \supp(W_j)$ with $\varphi(g)=\varphi(h)$, then  
			all terms from $U$ equal to $\varphi(g)$ are equal.
		\end{enumerate}
	\end{claim}
	\begin{proof}[Proof of \ref{sch-A1}]
		1. We have $W_{\sigma}:=\sigma(W_0)\cdot\ldots\cdot \sigma(W_{km-1})$ is a sequence over $\ker(\varphi)$ of length $km=\mathsf D_k(\ker(\varphi))$. Since $U\in \mathcal M_k(G)$, we have $W_{\sigma}\in \mathcal M_k(\ker(\varphi))$.
		It follows from Theorem \ref{le-cyclic} that $W_{\sigma}=\sigma(W_0)^{km}$ with $\sigma(W_0)$  a generator of $\ker(\varphi)$.
		
		Assume to the contrary that there exists some $i\in [0, km-1]$ such that $\varphi(W_i)$ is not minimal. Then we can decompose $W_i=W_i^{(1)}W_i^{(2)}$ such that both $\varphi(W_i^{(1)})$ and  $\varphi(W_i^{(2)})$ are nontrivial zero-sum. It follows that $W_{\sigma}^*:=\sigma({W_i})^{-1}\sigma(W_i^{(1)})\sigma(W_i^{(2)})W_{\sigma}$ is a sequence over $\ker(\varphi)$ of length $km+1>\mathsf D_k(\ker(\varphi))$, whence $W_{\sigma}^*\not\in \mathcal M_k(\ker(\varphi))$, a contradiction to  $U\in \mathcal M_k(G)$.
		
		2. Suppose there exist subsequences $S\t W_i$ and $T\t W_j$ with $i\neq j$ such that $\sigma(\varphi(S))=\sigma(\varphi(T))$. Then 
		 we can define $W'_i=S^{-1} W_i T$ and $W'_j=T^{-1} W_j S$. Setting $W'_s=W_s$ for all $s\neq i,j$, we then obtain a new  block decomposition $W'=(W'_0,W'_1,\ldots,W'_{km-1})$ with associated sequence $W'_\sigma=\sigma(W_i)^{-1} \sigma(W_j)^{-1} W_\sigma \sigma(W'_i) \sigma(W'_j)$. Since $k\ge 2$ and $m\ge 2$, we have $km-1\ge 3$ and it follows by Item 1 that $W'_{\sigma}=\sigma(W_s)^{km}$ for some $s\neq i,j$, whence $W'_{\sigma}=W_{\sigma}$. Therefore $\sigma(W_i')=\sigma(W_i)$, ensuring $\sigma(S)=\sigma(T)$.
		
		3. Suppose there are distinct blocks $W_i$ and $W_j$ having terms $g\in \supp(W_i)$ and $h\in \supp(W_j)$ with $\varphi(g)=\varphi(h)$. It follows by Item 2  that $g=h$. In such case, the assertion follows by doing this for all $g$ and $h$ contained in distinct blocks with $\varphi(g)=\varphi(h)$.
		\qedhere (\ref{sch-A1})
	\end{proof}

	  Since $U\in \mathcal M_k(G)$, we have  $\varphi(U)\in \mathcal M_{km}(\varphi(G))$ by Lemma \ref{le-epi}.
	 In view of Proposition \ref{pr-key}, we have that
	\be\label{notn-1} 0\notin \Sigma_{\leq n-1}(\varphi(G)).\ee
	Hence, since $\varphi(G)\cong C_n\oplus C_n$, we conclude from Lemma \ref{le-key} that there is some basis $(\overline e_1,\overline e_2)$ for $\varphi(G)\cong C_n\oplus C_n$ such that either
	\be\label{1st-popsicle} \supp(\varphi(U))\subset \{\overline e_1\}\cup\big(\la \overline e_1\ra+\overline e_2\big),\ee or else
	\be\label{2nd-popsicle} \varphi(U)=(\overline e_1)^{an} (\overline e_2)^{bn-1} (u\overline e_1+\overline e_2)^{cn-1} (u\overline e_1+2\overline e_2),\ee for some $u\in [2,n-2]$ with $\gcd(u,n)=1$, and some $a,b,c\geq 1$.

  We distinguish two cases depending on whether \eqref{1st-popsicle} or \eqref{2nd-popsicle} holds.

	\smallskip
	\noindent
	CASE 1: $\varphi(U)=(\overline e_1)^{an} (\overline e_2)^{bn-1} (u\overline e_1+\overline e_2)^{cn-1} (u\overline e_1+2\overline e_2),$ for some $u\in [2,n-2]$ with $\gcd(u,n)=1$, and some $a,b,c\geq 1$.

	Since $u\in [2,n-2]$ with $\gcd(u,n)=1$, it follows that   $n\geq 5$.
	In view of the hypothesis of CASE 1, we have  $(a+b+c)n-1=|U|=kmn+n-1$, implying \be\label{abcs}a+b+c=km+1.\ee Set $$\overline e_3=u\overline e_1+\overline e_2, \quad\mbox{ so } \quad \overline e_2=(n-u)\overline e_1+\overline e_3,$$
	and note that $\overline e_1=u^*(\overline e_2-\overline e_3)$, where $u^*\in [2,n-2]$ is the multiplicative inverse of $-u$ modulo $n$, so $$u^*u\equiv -1\mod n \quad\mbox{ with $u^*\in [2,n-2]$}.$$

	In view of the hypothesis of CASE 1, there is a block decomposition $W=(W_0,W_1,\ldots, W_{km-1})$ of $U$ with  \begin{align*} \varphi(W_0)=(\overline e_1)^{n-1} (\overline e_2)^{u^*} (\overline e_3)^{n-u^*},\quad
		&\varphi(W_1)=(\overline e_3)^{u^*-1} (\overline e_2)^{n-u^*-1} \overline e_1 (\overline e_2+\overline e_3),\quad \und\\
		&\varphi(W_i)\in \{(\overline e_1)^{n},\,(\overline e_2)^{n},\, (\overline e_3)^{n}\} \quad\mbox {for $i\in [2,km-1]$}.\end{align*}
	Let $z\in \supp(\varphi(U))=\{\overline e_1,\overline e_2,\overline e_3, \overline e_2+\overline e_3\}$. If $z=\overline e_2+\overline e_3$, then we trivially have $g=h$ for all $g,h\in\supp(U)$ with $\varphi(g)=\varphi(h)=z$, since there is a unique term $g\in \supp(U)$ with $\varphi(g)=z$. Otherwise, since
	$\vp_{\overline e_j}(W_i)>0$ for all $j\in[1,3]$ and $i\in [0,1]$, it follows
	that there are distinct block $W_i$ and $W_j$, for some $i,j\in [0,km-1]$, with terms $g\in \supp(W_i)$ and $h\in \supp(W_j)$ such that $\varphi(g)=\varphi(h)=z$. Then \ref{sch-A1}.3 implies 
	$$g,h\in \supp(U)\; \mbox{ with $\varphi(g)=\varphi(h)$ \; implies $g=h$}.$$
	As a result, we can find representatives $e_1$ and $e_2$ for $\overline e_1$ and $\overline e_2$, and $\alpha,\beta\in \ker\varphi$, such that
	$$\supp(U)=\{e_1, \ e_2, \ e_3+\alpha, \ e_2+e_3+\beta\},$$
	where $e_3:= ue_1+e_2$, \ $\varphi(e_1)=\overline e_1$, \ $\varphi(e_2)=\overline e_2$, \ $\varphi(e_3+\alpha)=\overline e_3=u\overline e_1+\overline e_2$, and $\varphi(e_2+e_3+\beta)=\overline e_2+\overline e_3=u\overline e_1+2\overline e_2$.
	
	Since $u,u^*\in [2,n-2]$, it follows that there are subsequences $e_1^u e_2\mid W_0$ and $e_3+\alpha=ue_1+e_2+\alpha\mid W_1$. By \ref{sch-A1}.2, we have $ue_1+e_2=\sigma(e_1^u e_2)=ue_1+e_2+\alpha$, whence $\alpha=0$. Likewise, there are subsequences $e_1^u e_2^2\mid W_0$ and $e_2+e_3+\beta=ue_1+2e_2+\beta\mid W_1$. By \ref{sch-A1}.2, we have $ue_1+2e_2=\sigma(e_1^u e_2^2)=ue_1+2e_2+\beta$, whence $\beta=0$. As a result, $\supp(U)=\{e_1, \ e_2, \ ue_1+e_2, \ ue_1+2e_2\}$, which together with the hypotheses of CASE 1 gives
	\begin{align}\label{U-go}&U=e_1^{an} e_2^{bn-1} (ue_1+e_2)^{cn-1} (ue_1+2e_2),\\
		&W_0=e_1^{n-1} e_2^{u^*} (ue_1+e_2)^{n-u^*}\;\und\; W_1=e_1 e_2^{n-u^*-1} (ue_1+e_2)^{u^*-1} (ue_1+2e_2).\nn
	\end{align}
	
	From \eqref{U-go}, we have $\supp(U)\subset \la e_1,e_2\ra$. If this were a proper subgroup of $G=C_{n}\oplus C_{mn}$, then  $\mathsf D_k(G)=|U|\le \mathsf D_k(\la e_1,e_2\ra)<\mathsf D_k(G)$ (by Lemma \ref{le-subgroup}), a contradiction. Therefore we instead conclude that \be\label{keyboard}\la e_1,e_2\ra=G=C_n\oplus C_{mn}.\ee

	If $T\mid W_0 W_1$ is any proper, nontrivial subsequence with $\varphi(T)$ zero-sum, then we can set $W'_0=T$, define $W'_1$ by $W'_0 W'_1=W_0 W_1$ and set $W'_i=W_i$ for all $i\geq 2$ to thereby obtain a new block decomposition $W'$. Since $km\ge 4$, \ref{sch-A1}.1 ensures that $\sigma(W'_0)=g_0$, where $g_0:=\sigma(W_0)$ is a generator for $\ker\varphi\cong C_m$.
	This shows that  \be\label{sticky}\mbox{any proper, nontrivial subsequence $T\mid W_0 W_1$ with $\varphi(T)$ zero-sum has $\sigma(T)=g_0:=\sigma(W_0)$}.\ee
	In particular, since $e_1^n\mid W_0 W_1$  and $e_1^{n-u} e_2^{n-1} (ue_1+e_2)\mid W_0 W_1$, we have $$ ne_1=\sigma(e_1^{n-u} e_2^{n-1} (ue_1+e_2))=ne_1+ne_2=g_0,$$ 
	forcing $ne_2=0$. In view of $\ord(\overline e_2)=n$, we have $\ord(e_2)=n$.
	 Since $\mathsf v_{e_2}(U)=bn-1\ge n-1$, it follows from Lemma \ref{sch-I}.1 that $U$ has the form of type I, a contradiction to our assumption of \eqref{assume}.

	\smallskip
	\noindent
	CASE 2: $\supp(\varphi(U))\subset \{\overline e_1\}\cup\big(\la \overline e_1\ra+\overline e_2\big)$.
	
Let $W=(W_0,\ldots,W_{km-1})$ be a block decomposition of $U$.	We say $W$ is {\bf refined} if  $|W_i|\leq n$ for each $i\in [1,km-1]$.
Since $|U|=(km-2)n+3n-1\geq (km-2)n+3n-2=(km-2)n+\eta(\varphi(G))$ with $\sigma(U)=0$ by Theorem \ref{th-gen}, repeated application of the definition $\eta(\varphi(G))$ to the sequence $\varphi(U)$ shows that $U$ has a refined block decomposition.

 Let $W=(W_0,\ldots,W_{km-1})$ be a refined block decomposition of $U$. In view of \ref{sch-A1}.1, we have
 $\varphi(W_0)$ is a minimal zero-sum sequence of terms from $\varphi(G)\cong C_n\oplus C_n$, thus with $|W_0|\leq \mathsf D(C_n\oplus C_n)=2n-1$. Since each $|W_i|\leq n$ for $i\in [1,km-1]$, we have $2n-1\geq |W_0|=|U|-\Sum{i=1}{km-1}|W_i|\geq kmn+n-1-(km-1)n=2n-1$, forcing equality to hold in all estimates. In particular, we now conclude that
\begin{align}\label{blockhelp} |W_0|=2n-1\;\und\;|W_i|=n\quad\mbox{ for all $i\in [1,km-1]$},\end{align} for any refined block decomposition $W$, with $\varphi(W_0)$ always a minimal zero-sum of length $2n-1$.
	
 In view of the hypothesis of CASE 2, any zero-sum subsequence of $\varphi(U)$ must have the number of  terms from $\la \overline e_1\ra+\overline e_2$  congruent to $0$ modulo $n$. In particular,
	\begin{align}\label{W-cape}\varphi(W_0)=(\overline e_1)^{n-1}
		\prod_{i\in [1,n]}(-x_i\overline e_1+\overline e_2),\end{align} for some  $x_1,\ldots,x_n\in  \Z$  with $x_1+\ldots+x_n\equiv n-1\mod n$
	and, for every $j\in[1,km-1]$, either \begin{align}\label{steer}
		&\varphi(W_j)=(\overline e_1)^n\;\mbox{ or }\; \varphi(W_j)=\prod_{i\in [1,n]}(-y_{i}\overline e_1+\overline e_2),\end{align} for some $y_{1},\ldots,y_{n}\in \Z$ with   $y_{1}+\ldots+y_{n}\equiv 0\mod n$.

	\begin{claim}
		\label{sch-A2} There is some $e_1\in G$ such that every $g\in \supp(U)$ with $\varphi(g)=\overline e_1$ has $g=e_1$.
	\end{claim}
	
	\begin{proof}[Proof of \ref{sch-A2}]
	Let $W$ be a refined block decomposition.
	Then \eqref{W-cape} implies that $\vp_{\overline e_1}(\varphi(W_0))=n-1\geq 1$. If there exists $i\in [1, km-1]$ such that $\overline e_1\in \supp(\varphi(W_i))$, then the assertion follows by \ref{sch-A1}.3.
	Thus we may assume that $\supp(W_i)\subset \la \overline e_1\ra+\overline e_2$ for all $i\in [1,km-1]$.
	If $n=2$, the assertion is trivial. Suppose $n\ge 3$.  
		
		Let $i\in [1,km-1]$ and $h\in \supp(W_i)$ be arbitrary, say with $\varphi(h)=y\overline e_1+\overline e_2$. Suppose there is some $g\in \supp(W_0)$ with $\varphi(g)=x\overline e_1+\overline e_2$ and $x\notin \{y,y+1\}\mod n$. Then, letting $z\in [1,n-2]$ be the integer such that $z+x\equiv y\mod n$, it follows that there is  a subsequence $T g\mid W_0$ with $\varphi(T g)=(\overline e_1)^z (x\overline e_1+\overline e_2)$ and $\sigma(\varphi(T g))=(z+x)\overline e_1+\overline e_2=y\overline e_1+\overline e_2=\varphi(h)$.
		Note \ref{sch-A1}.2 implies that $$\sigma(T)+\sigma(g)=\sigma(h).$$ Since $|T|=z\in [1,n-2]$, there are terms $f_1\in \supp(T)$ and $f_2\in \supp(T^{-1} W_0)$ with $\varphi(f_1)=\varphi(f_2)=\overline e_1$.
		Repeating this argument using the subsequence $T'=f_1^{-1} T f_2$ in place of $T$, we again find that $$f_2-f_1+\sigma(h)=f_2-f_1+\sigma(T)+\sigma(g)=\sigma(T')+\sigma(g)=\sigma(h),$$ implying that $f_1=f_2$. Doing this for all $f_1\in \supp(T)$ and $f_2\in \supp(T^{-1} W_0)$ with $\varphi(f_1)=\varphi(f_2)=\overline e_1$ would then yield the desired conclusion for \ref{sch-A2}. So we can instead assume every $g\in\supp(W_0)$ with $\varphi(g)\in \la \overline e_1\ra+\overline e_2$ has either $\varphi(g)=y\overline e_1+\overline e_2$ or
		$\varphi(g)=(y+1)\overline e_1+\overline e_2$. In view of \eqref{W-cape}, both possibilities must occur (since $x_1+\ldots+x_n\equiv 1\mod n$ in \eqref{W-cape}), forcing \be\label{force1}\supp(\varphi(W_0))=\{\overline e_1,\ y\overline e_1+\overline e_2, \ (y+1)\overline e_1+\overline e_2\}.\ee
		Moreover, the above must be true for any $i\in [1,km-1]$ and $h\in \supp(W_i)$. Since $n\geq 3$, the value of $y$ is uniquely forced by \eqref{force1}, which means  \be\label{morgor}\varphi(W_i)=(y\overline e_1+\overline e_2)^n\quad\mbox{ for all $i\in [1,km-1]$}.\ee
		
		Suppose there are two terms $g_1 g_2\mid W_0$ with $\varphi(g_1)=\varphi(g_2)=(y+1)\overline e_1+\overline e_2$. Let $h_1 h_2\mid W_1$ be a length two subsequence. Then there is a subsequence $T g_1 g_2\mid W_0$ with $\varphi(T g_1 g_2)=(\overline e_1)^{n-2} ((y+1)\overline e_1+\overline e_2)^2$. Thus $\sigma(\varphi(T g_1 g_2))=\sigma(\varphi(h_1 h_2))$ and  by \ref{sch-A1}.2, we conclude that $\sigma(T)+g_1+g_2=h_1+h_2$. Since $1\leq |T|=n-2<n-1$, we can find $f_1\in \supp(T)$ and $f_2\in \supp(T^{-1} W_0)$ with $\varphi(f_1)=\varphi(f_2)=\overline e_1$ and argue as before to conclude that $f_1=f_2$. Doing this for all $f_1$ and $f_2$ then yields the desired conclusion for \ref{sch-A2}. So we can instead assume that $\vp_{(y+1)\overline e_1+\overline e_2}(\varphi(U))=1$, implying via \eqref{force1} and \eqref{W-cape} that
		$\vp_{y\overline e_1+\overline e_2}(\varphi(W_0))=n-1$. Combined with \eqref{morgor}, we find that $\vp_{y\overline e_1+\overline e_2}(\varphi(U))=kmn-1$. Moreover, by \ref{sch-A1}.3, all $kmn-1$ of the corresponding terms from $U$ must be equal to the same element (say) $g_0\in G$, whence $\vp_{g_0}(U)\geq kmn-1$, forcing $\ord(g_0)=mn$ else  $\vp_{g_0}(U)\ge (k+1)\ord(g_0)$ implies $U$ contains $k+1$ disjoint zero-sum subsequences of length $\ord(g_0)$, contradicting the hypothesis that $U\in \mathcal M_k(G)$. Applying Lemma \ref{sch-I} shows  $U$ has the form of type II, a contradiction to our assumption \eqref{assume}.\qedhere (\ref{sch-A2})
	\end{proof}

	In view of \ref{sch-A2},
	we can decompose $$U=e_1^{\mathsf v_{\overline e_1}(\varphi(U))} U^*$$ with  $U^*\mid U$ the subsequence consisting of all terms $g$ with $\varphi(g)\in\la \overline e_1\ra+\overline e_2$ (view those as $U^*$-terms).	
	Note that $\vp_{e_1}(U)\geq \vp_{e_1}(W_0)= n-1$. If $\ord(e_1)=n$, then Lemma \ref{sch-I}.1 implies that $U$ has the form of type I, a contradiction to our assumption \eqref{assume}. Thus $\ord(e_1)>n$. Since $\varphi(e_1)=\overline e_1$ with $\ord(\overline e_1)=n$, it follows that $\ord(e_1)$ is a  multiple of $n$. Thus
	\be\label{ge_2n}\ord(e_1)\geq 2n.\ee
	
	If $\varphi(W_j)=(\overline e_1)^n$ for all $j\in [1, km-1]$, then 
	$\vp_{e_1}(U)=n-1+(km-1)n=kmn-1$ follows from \ref{sch-A2}. Since $U\in \mathcal M_k(G)$, we have $\vp_{e_1}(U)=kmn-1\le k\ord(e_1)$, ensuring $\ord(e_1)=mn$. Now Lemma \ref{sch-I}.2 implies that $U$ has the from of type II, a contradiction to our assumption \eqref{assume}.
	Thus we can assume 
	\be\label{at_least_one}
	\text{there is at least one block $W_j$ with $j\geq 1$ containing some $U^*$-term.}
	\ee
	 Let $e_2\t W_j$ be some $U^*$-term.
	 Then  $\varphi(e_2)\in \la \overline e_1\ra+\overline e_2$.
	We note that the hypotheses of CASE 2 hold with the basis $(\overline e_1,\overline e_2)$ replaced by the basis $(\varphi (e_1),\varphi (e_2))$. Thus, by replacing $\overline e_2$ by  $\varphi (e_2)$, we may assume that $\varphi(e_2)=\overline e_2$.
	Let $I=[0,n-1]$ be the discrete interval of length $n$.  Each $U^*$-term $g$  can be written uniquely as $g=-\iota(g)e_1+e_2+\psi(g)$ for some $\iota(g)\in I\subset \Z$ and $\psi(g)\in \ker\varphi$.

	\begin{claim}
		\label{sch-A3} Suppose $W=(W_0,\ldots, W_{km-1})$ is a  refined block decomposition for $U$. If $g\in \supp(W_0)$ and $h\in \supp(W_j)$ are  $U^*$-terms, where $j\geq 1$, then
		$$\psi(g)-\psi(h)=\left\{
		\begin{array}{ll}
			0 & \hbox{if $\iota(g)\geq \iota(h)$,} \\
			-ne_1 & \hbox{if $\iota(g)<\iota(h)$.}
		\end{array}
		\right.$$
	\end{claim}

	\begin{proof}[Proof of \ref{sch-A3}]
		Let $\iota(g)=x$ and  $\iota(h)=y$,  and let $\psi(g)=\alpha$ and $\psi(h)=\beta$. Then $$g=-xe_1+e_2+\alpha\;\und\; h=-ye_1+e_2+\beta.$$
		If $x\geq y$, let $z=x-y\in [0,n-1]$. If $x<y$, let $z=x-y+n\in [1,n-1]$.
		In both cases, we have $\sigma(\varphi(e_1^{z} g))=(z-x)\overline e_1+\overline e_2=-y\overline e_1+\overline e_2=\varphi(h)$, so  \ref{sch-A1}.2 implies  that $(z-x)e_1+e_2+\alpha=\sigma(e_1^z g)=h=-ye_1+e_2+\beta$, whence $$\alpha-\beta=(x-y-z)e_1.$$
		If $x\geq y$, then $z=x-y$, implying $\psi(g)-\psi(h)=\alpha-\beta=0$. If $x< y$, then $z=x-y+n$, implying $\psi(g)-\psi(h)=\alpha-\beta=-ne_1$.\qedhere (\ref{sch-A3})
	\end{proof}
	
	Note that $e_2$ is a $U^*$-term of some $W_j$ with $\iota(e_2)=0$ and $\psi(e_2)=0$. Let $W_0^*\t W_0$ be the subsequence consisting of all $U^*$-terms. 
	Thus for every term $g$ of $W_0^*$, we have $\iota(g)\ge 0=\iota(e_2)$ and hence \ref{sch-A3} implies that $\psi(g)=\psi(e_2)=0$, that is, 
	\be\label{psi-W0}
	\text{ for every term $g$ of $W_0^*$, we have $\psi(g)=0$.}
	\ee

	Let $h$ be a $U^*$-term of $W_0^{-1}U$. Then there exists $j\in [1, km-1]$ such that $h\t W_j$. By \eqref{steer}, we have $\iota(h^{-1} W_0^{*} W_j)\mod n$ is a sequence of $2n-1$ terms from a cyclic group of order $n$, thus containing a zero-sum sequence of length $n$, say $\sigma(\iota(W'_j))\equiv 0\mod n$ with $W'_j\mid h^{-1} W_0^* W_j$ and  $|W'_j|=n$. Define $W'_0$ by $W'_0 W'_j=W_0 W_j$ and set $W'_i=W_i$ for all $i\neq 0,j$. Then $W'=(W'_0,W'_1,\ldots,W'_{km-1})$ is a refined block decomposition of $U$ with $h\in\supp(W'_0)$ by construction. Since $h\in \supp(W'_0)$, we must have $g\in\supp(W'_j)$ for some $g\in \supp(W_0^{*})$.
	If $\iota(g)>\iota(h)$, then, applying \ref{sch-A3} to the block decomposition $W$ and $W'$, it follows that $\psi(g)-\psi(h)=0$ and $\psi(h)-\psi(g)=-ne_1$, a contradiction to \eqref{ge_2n}.
	If $\iota(g)<\iota(h)$,  then, applying \ref{sch-A3} to the block decomposition $W$ and $W'$, it follows that $\psi(g)-\psi(h)=-ne_1$ and $\psi(h)-\psi(g)=0$, a contradiction to \eqref{ge_2n}. Therefore $\iota(g)=\iota(h)$ and by applying \ref{sch-A3} to the block decomposition $W$, we obtain $\psi(h)=\psi(g)=0$ by \eqref{psi-W0}. Therefore
	\be\label{iota}
	\begin{aligned}
	&\text{ for every $U^*$-term $h$ of $W_0^{-1}U$,we have $\psi(h)=0$ and }\\
	& \text{ there exists a term $g$ of $W_0^*$ such that $\iota(g)=\iota(h)$.}
	\end{aligned}
	\ee
	
	Note $e_2$ is a $U^*$-term of $W_0^{-1}U$. Then \eqref{iota} implies that there exists $g\t W_0^*$ such that $\iota(g)=\iota(e_2)=0$. 
Assume that there exists a $U^*$-term $h$ of $W_0^{-1}U$ such that $\iota(h)>0=\iota(g)$. In view of \eqref{iota} and \eqref{psi-W0}, we have $\psi(h)=\psi(g)=0$.
 By applying \ref{sch-A3} to the block decomposition $W$, we obtain $0=\psi(g)-\psi(h)=-ne_1$, a contradiction to \eqref{ge_2n}.
Thus for every $U^*$-term $h$ of $W_0^{-1}U$, we have $\iota(h)=0$
and combined with \eqref{iota}, we obtain $h=e_2$.
Therefore, for every $j\in [1,km-1]$, we either have \be\label{reeper}W_j=e_1^n\;\mbox{ or } \;W_j=e_2^n.\ee
	
	In view of \eqref{at_least_one}, we 
	let $$s\in [0,km-2]$$ be the number of blocks $W_j$ equal to $e_1^n$. It follows from \eqref{W-cape} and \eqref{steer} that
	\be\label{Ustarter}U=e_1^{(s+1)n-1} e_2^{(km-s-1)n}
	\prod_{i\in [1,n]}(-x_i e_1+e_2)\ee for some $x_1,\ldots, x_{n}\in [0,n-1]$ with $x_1+\ldots+ x_{n}\equiv n-1\mod n$. 
	
	 Assume that $x_1+\ldots+x_n\neq n-1$. Then $x_1+\ldots+x_n\ge  2n-1$ and hence 
	 there will be some minimal index $t\in [2,n-1]$ such that $x_1+\ldots+x_t\geq n$. By the minimality of $t$, we have $x_1+\ldots+x_{t-1}\leq n-1$, which combined with $x_t\in [1,n-1]$ ensures that $x_1+\ldots+x_{t}\in [1,2n-2]$. Hence $x_1+\ldots+x_t=n+r$ for some $r\in [0,n-2]$. Let  $j\in [1,km-1]$ such that $W_j=e_2^n$ (which exists by \eqref{at_least_one}). Since $\sigma\big(\varphi\big(e_1^r \prod_{i\in [1,t]}(-x_ie_1+e_2)\big)\big)=t\overline e_2=\sigma(\varphi(e_2^t))$, it follows from \ref{sch-A1}.2 that $-ne_1+te_2=\sigma(e_1^r \prod_{i\in [1,t]}(-x_1e_1+e_2))=\sigma(e_2^t)=te_2$, whence $ne_1=0$, contradicting  \eqref{ge_2n}. So we instead conclude that
	\be\nn x_1+\ldots+x_n=n-1.\ee
	
	As already noted, there is some block $W_j=e_2^n$, with $j\in [1,km-1]$. By \ref{sch-A1}.1, we obtain $ne_2=\sigma(W_j)=g_0$ is a generator for $\ker(\varphi)\cong C_m$, ensuring that $\ord(ne_2)=\ord(g_0)=m$. We also have $\varphi(e_2)=\overline e_2$ with $\ord(\overline e_2)=n$, ensuring that $\ord(e_2)=n\ord(ne_2)=nm=\exp(G)$. 
	Since $\supp(U)\subset \la e_1,e_2\ra$, we obtain that 
	$|U|=\mathsf D_k(G)\le \mathsf D_k(\la e_1,e_2\ra)\le \mathsf D_k(G)$. It follows from Lemma \ref{le-subgroup} that $G=\langle e_1,e_2\rangle$ and hence $(e_1,e_2)$ is a generating set of $G$ with $\ord(e_1)>n$ and $\ord(e_2)=mn$.

	If $s=0$, then  $U$ has the form of type III by writing $e_2$ as $-0e_1+e_2$.
	Suppose $s\geq 1$. Then, in view of \eqref{reeper}, there is some block $W_i=e_1^n$ with $i\in [1,km-1]$, while there is some block $W_j=e_2^n$ with $j\in [1,km-1]$. By \ref{sch-A1}.1, we have $ne_2=\sigma(W_j)=\sigma(W_i)=ne_1$. 
	Since $\ord(e_1)$ is a multiple of $n$, we have $\ord(e_1)=n\ord(ne_1)=n\ord(ne_2)=mn$. Note that $(e_1, e_2-e_1)$ is a generating set of $G$. It follows from $n(e_2-e_1)=0$ that $(e_2-e_1, e_1)$ is a basis of $G$. Letting $f_1=e_1$ and $f_2=e_2-e_1$, we have 
	\[
	U=f_1^{(s+1)n-1}(f_1+f_2)^{(km-s-1)n}\prod_{i=1}^n((1-x_i)f_1+f_2)\quad \text{ with }s+1\in [2, km-1]
	\]
	and hence   $U$ has the form of type IV by writing $f_1+f_2$ as $(1-0)f_1+f_2$.
\end{proof}

\end{document}